\documentclass[12pt,twoside,a4paper]{amsart}
\usepackage{amssymb}
\date{\today}

\def\nbh{neighborhood }

\def\End{{\rm End}}
\def\ann{{\rm ann}}
\def\deg{\text{deg}\,}

\def\w{\wedge}

\def\dbar{\bar\partial}

\def\C{{\mathbb C}}
\def\w{{\wedge}}

\def\M{{\mathcal M}}

\def\S{{\mathcal S}}

\def\CH{\mathcal{CH}}

\def\Hom{{\rm Hom\, }}

\def\codim{{\rm codim\,}}

\def\Ok{{\mathcal O}}

\def\Re{{\rm Re\,  }}

\def\ann{{\rm ann\,}}

\def\PM{{\mathcal{PM}}}
\def\Cu{{\mathcal C}}
\def\Homs{{\mathcal Hom\, }}
\def\J{{\mathcal J}}
\def\I{{\mathcal I}}
\def\be{\begin{equation}}
\def\ee{\end{equation}}

\newtheorem{thm}{Theorem}[section]
\newtheorem{lma}[thm]{Lemma}
\newtheorem{cor}[thm]{Corollary}
\newtheorem{prop}[thm]{Proposition}

\theoremstyle{definition}

\theoremstyle{remark}

\newtheorem{preremark}{Remark}
\newtheorem{preex}{Example}

\newenvironment{remark}{\begin{preremark}}{\qed\end{preremark}}
\newenvironment{ex}{\begin{preex}}{\qed\end{preex}}

\numberwithin{equation}{section}

\title[]{A residue criterion for strong holomorphicity}

\begin{document}

\date{\today}

\author{Mats Andersson}

\address{Department of Mathematics\\Chalmers University of Technology and the University of 
G\"oteborg\\S-412 96 G\"OTEBORG\\SWEDEN}

\email{matsa@math.chalmers.se}


\thanks{The author was
  partially supported by the Swedish Natural Science
  Research Council}

\begin{abstract}
We give  a local criterion in terms of a residue current for strong holomorphicity
of a meromorphic function on an arbitrary pure-dimensional
analytic variety. This
generalizes a result by A.\ Tsikh for the case of a reduced
complete intersection. 
\end{abstract}



\maketitle

\section{Introduction}

Let $Z$ be an analytic variety in a \nbh of the closed unit ball  in $\C^n$,
and let $\I_Z$ be the sheaf of holomorphic functions that vanish on $Z$.  Then 
$\Ok_Z=\Ok/\I_Z$ is the sheaf of (strongly) holomorphic functions on $Z$. 
A meromorphic function on $Z$ is a section of the sheaf  $\M_Z$, where
$\M_{Z,x}$ is the ring of quotients  $g/h$, where $g,h\in\Ok_{Z,x}$  and $h$
is a nonzerodivisor. Thus  locally a meromorphic function $\phi$
is   (represented by)   $g/h$ where $g,h$ are holomorphic in the ambient space 
and $h$ is generically non-vanishing on $Z$, and 
$g'/h'$ is another representation of $\phi$ if and only if $gh'=g'h$ 
on $Z$.

If  $Z$ is given by a complete intersection, 
i.e., $Z=\{F_1=\cdots =F_p=0\}$ and $\codim Z=p$,   we  have a well-defined 
$\dbar$-closed $(0,p)$-current
$$
\mu^F=\dbar\frac{1}{F_p}\w\ldots\w\dbar\frac{1}{F_1},
$$
the Coleff-Herrera product, \cite{CH},   with support on $Z$.
The following
criterion was proved by A.\ Tsikh, \cite{Ts};  see also \cite{HP}:
\smallskip

{\it 
Assume that the Jacobian
$dF_1\w\ldots\w dF_p$ is  non-vanishing on $Z_{reg}$.
A meromorphic function $\phi$ on $Z$ is (strongly) holomorphic on $Z$ if and only if the
current 
$
\phi\mu^F
$
is $\dbar$-closed.}
\smallskip

The assumption on the Jacobian implies (and is in fact 
equivalent to) that the annihilator of $\mu^F$ is
precisely $\I_Z$.
The product $\phi\mu^F$ can be defined as the principal value
\begin{equation}\label{ep}
\lim_{\epsilon\to 0}\chi(|h|/\epsilon) (g/h)\mu^F,
\end{equation}
where $g/h$ is a (local) representation of $\phi$ and $\chi$ is
(a possibly smooth approximand of) the characteristic function for 
the interval $[1,\infty)$, see Section~\ref{mult}.  
For  further reference let us sketch a proof of  Tsikh's theorem:   
If $\phi$ is strongly holomorphic, then it is represented
by a function $\Phi$  that is  holomorphic in the ambient space, and since
$\mu^F$ is $\dbar$-closed it follows that $\phi\mu^F$ is.
Conversely,  assume that $\phi=g/h$ where $g,h$ are
holomorphic in the ambient space (and necessarily) $h$ is generically non-vanishing 
on $Z_{reg}$. 
Then formally at least, the assumption implies that
$$
g\dbar\frac{1}{h}\w\dbar\frac{1}{F_p}\w\ldots\w\dbar\frac{1}{F_1}=0,
$$
and since also $h,F_1,\ldots, F_p$ form  a complete intersection it follows from the
duality theorem,  \cite{DS} and \cite{P}, that $g$ is in the ideal
generated by $h, F_1,\ldots, F_p$, i.e., $g=\alpha h+\alpha_1F_1+\cdots +\alpha_pF_p$.
Thus  $\phi$ is represented by $\alpha\in\Ok$ and  so  $\phi\in\Ok_Z$.

\begin{remark}
One should remark here that it is {\it not} possible
to use the Lelong current $[Z]$; in fact,  the meromorphic functions $\phi$ such that
$\phi[Z]$ are  $\dbar$-closed,  form the wider class
$\omega^0_Z$  introduced  in \cite{Bar} and studied further in \cite{HP}.
\end{remark}

In this paper we generalize Tsikh's  result in two ways.   We consider
an arbitrary  variety $Z$ of pure codimension $p$, and we consider also
the non-reduced case, i.e., instead of $\I_Z$ we have  an arbitrary  
pure-dimensional coherent ideal sheaf $\I$ with zero variety $Z$.
To formulate our results we first have to discuss
an appropriate generalization from \cite{AW1}
of the   Coleff-Herrera product above.

In a \nbh $X$ of the closed unit ball
there is a free resolution
\begin{equation}\label{acomplex}
0\to \Ok(E_N)\stackrel{f_N}{\longrightarrow}\ldots
\stackrel{f_3}{\longrightarrow} \Ok(E_2)\stackrel{f_2}{\longrightarrow}
\Ok(E_1)\stackrel{f_1}{\longrightarrow}\Ok(E_0)
\end{equation}
of  the sheaf $\Ok/\I$. 
Here $\Ok(E_k)$ is the free sheaf
associated to the trivial vector bundle $E_k$ over $X$,
and $E_0\simeq \C$ so that $\Ok(E_0)\simeq\Ok$. 
In \cite{AW1} we defined, given Hermitian metrics
on $E_k$,  a residue current $R=R_p+R_{p+1}+\cdots $  with support on $Z$, 
where $R_k$ is a $(0,k)$-current that takes values in
$E_k\simeq\Hom(E_0,E_k)$, such that a holomorphic function 
$\phi$ is in $\I$ if and only if  $\phi R=0$.  
For simplicity we think that we have some fixed global frames for
$E_k$ and choose the trivial metrics that they induce. In this way we can talk about
{\it the} residue current associated with \eqref{acomplex}.

If $\I$ is Cohen-Macaulay, i.e.,
each stalk $\I_x$ is a Cohen-Macaulay ideal in $\Ok_x$
we can choose \eqref{acomplex} such that  $N=p$,  and then $R=R_p$ is 
$\dbar$-closed. In general, 
$f_{k+1}R_{k+1}-\dbar R_k=0$ for each $k$ which can be written
simply as $\nabla R=0$ if   $\nabla=f-\dbar$ and $f=\oplus f_k$.

The assumption that $\I$ has pure dimension $p$ means that 
in each local ring $\Ok_x$  all the associated primes
have codimension $p$. 
As in the reduced case we have $\Ok_Z=\Ok/\I$. 
The sheaf of meromorphic functions is defined in precisely
the same way as in the reduced case.
Thus, if   $\Phi$ and $\Phi'$ are meromorphic in the ambient space then
they define the same meromorphic $\phi$ on $Z$ if and only if
$\Phi-\Phi'$ belongs to $\I$ generically on $Z$.
In Section~\ref{mult} we give a reasonable definition of $\phi R$ for 
$\phi\in\M_Z$.   Here is our basic result.

\begin{thm}\label{thmett} 
Suppose that $Z\sim \I$ has pure codimension $p$ and let $R$ be the residue current 
associated to a resolution of $\Ok/\I$. Then a meromorphic function $\phi$ on
$Z$ is (strongly) holomorphic if and only if
\begin{equation}\label{nabvil}
\nabla (\phi R)=0.
\end{equation}
\end{thm}

If $\I$ is Cohen-Macaulay and $N=p$ in \eqref{acomplex}, then $R=R_p$ and
so \eqref{nabvil} means that $\dbar (\phi R)=0$.
\smallskip

The reduced case  of course corresponds to $\I=\I_Z$.

\begin{remark}
If $f_1=(F_1,\ldots, F_p)$ is a complete intersection,
one can choose \eqref{acomplex} as the Koszul complex, and  then the 
residue current is precisely the Coleff-Herrera product $\mu^F$, see, e.g.,
\cite{A12} Corollary~3.2.
If  $\I=\I_Z$  we thus get   back Tsikh's theorem.
\end{remark}

Let $\I$ be any ideal sheaf of codimension $p$ 
and let \eqref{acomplex} be a resolution
of $\Ok/\I$. Let $Z_k$ be the analytic set where  $f_k$ does not have
have optimal rank. These sets $Z_k$  are independent of the choice of resolution, 
$\subset Z_{p+2}\subset Z_{p+1}\subset Z_{sing}\subset Z_p=\cdots =Z_1=Z$, 
where $Z$ is  the zero set of $\I$, and $\codim Z_k\ge k$ for all $k$.
Moreover, $\I$  is pure if and only if
$\codim Z_k\ge k+1$ for all $k>p$,  and
$\I$ is Cohen-Macaulay if and only if $Z_k=\emptyset$ for $k>p$.
All these facts are well-known and can be found in, e.g., 
\cite{Eis1} Ch.~20.

\smallskip
For each  meromorphic function $\phi$ on $Z\sim\I$ there is a
smallest analytic subvariety $P_\phi$, the pole set,
outside which $\phi$ is strongly holomorphic.
As an application of Theorem~\ref{thmett} we get

\begin{thm}\label{thmtva}
Assume that $Z$  has pure codimension $p$. If   $\phi$ is meromorphic and 
\begin{equation}\label{kruka}
\codim(P_\phi\cap Z_k)\ge k+2,\quad  k\ge p,
\end{equation}
then $\phi$ is (strongly) holomorphic.
\end{thm}

Assume now that $Z$ is reduced.
Recall that a function is called {\it weakly holomorphic} on $Z$ if it is holomorphic on
$Z_{reg}$ and locally bounded at $Z_{sing}$. It is well-known that each weakly holomorphic
function is meromorphic, see, e.g., \cite{Dem}.
If each germ of a weakly holomorphic function at $x\in Z$
is strongly holomorphic, then necessarily $Z_x$ is irreducible and 
$x$ is said to be a normal point.  
If $\phi$ is weakly holomorphic, then clearly
$P_\phi$ is contained in $Z_{sing}$.
From Theorem~\ref{thmtva} we therefore immediately get

\begin{cor}\label{thmtre}
Assume that $Z$ is reduced with pure codimension $p$ and let $\I_x$ be the 
corresponding local ideal at $x\in Z$. If 
\begin{equation}\label{kruka1}
\codim Z_{sing,x}\ge 2+p,
\end{equation}
and 
\begin{equation}\label{kruka2}
\codim Z_{k,x}\ge 2+k,\quad  k>p,
\end{equation}
then $x$ is a normal point.
\end{cor}

Conversely,  the conditions \eqref{kruka1} and \eqref{kruka2} are fulfilled if
$x$ is a normal point. 
In fact, these conditions are equivalent to Serre's criterion
(conditions $R1$ and $S2$)
for the ring $\Ok_{Z,x}$ to be normal, see, e.g., \cite{Eis1} p 255 and 462.
(The condition  \eqref{kruka1} is precisely  $R1$ 
and by an argument similar to the proof of  Corollary~20.14 in \cite{Eis1} it follows that
\eqref{kruka2} is equivalent to the condition $S2$.)
The normality of  $\Ok_{Z,x}$ is  
equivalent to that it is equal to its integral closure in $\M_{Z,x}$, 
which in turn is equivalent
to that $x$ is a normal point, see also \cite{AK}.

\begin{remark} One can check that the sets $Z^0=Z_{sing}$ and
$Z^\ell=Z_{p+\ell}$ for $\ell>0$ are independent of the embedding and thus
intrinsic analytic subset of the analytic space $Z$. In this notation the Serre condition
says that  $\codim Z^\ell\ge 2+\ell$ for $\ell\ge 0$.
\end{remark}

\begin{ex}  If $\I_x$ is a Cohen-Macaulay ideal, the $Z_k=\emptyset$ for $k>p$ and hence
\eqref{kruka2} is trivially fulfilled. If $Z_{sing}$ is just a point $x$, then
\eqref{kruka2} is fulfilled  if $Z_k$ avoids $x$ for each $k>n-2$.
This means that  $\Ok_{Z,x}=\Ok_x/\I_x$ has depth
at least $2$.   
\end{ex}

We also obtain a new proof of the following result due to Malgrange \cite{Mal}
and Spallek \cite{Spall}. One says that a function $\phi$ on $Z$ is in
$C^k(Z)$ if it is (locally) the restriction to $Z$ of a $C^k$-function in the
ambient space.

\begin{cor}\label{ms}
Assume that $Z$ has pure codimension and is reduced.
There is a natural number $m$ such that 
if $\phi\in C^m(Z)$ is  holomorphic on $Z_{reg}$ then $\phi$ is
strongly holomorphic on $Z$.
\end{cor}

\smallskip

It is desirable to express the  ideal $\I$ as 
\begin{equation}\label{snitt}
\I=\cap_1^\nu \ann\mu_\ell,
\end{equation}
where $\mu_j$ are so-called Coleff-Herrera currents, $\mu_j\in\CH_Z$, on $Z$.
In fact,  (locally) a Coleff-Herrera current $\mu$ is just  a meromorphic
differential operator acting on the current of integration
$[Z]$ (combined with contractions with   holomorphic vector fields),
see \cite{JEB4} (or \cite{A11}). Therefore 
$\phi\mu=0$  is an elegant 
intrinsic way to express that  certain holomorphic differential
operators applied to  $\phi$ vanish on $Z$.
If $\I$ has pure codimension    then, see, e.g.,  (1.6) in \cite{A11},
 $\I$ is equal to the annihilator of the analytic sheaf 
$$
\Homs(\Ok/\I,\CH_Z)=\{\mu\in\CH_Z;\  \I\mu=0\}.
$$
This sheaf  turns out to be coherent, and therefore 
there is a finite family of global sections in a \nbh $X$ of the closed unit
ball such that \eqref{snitt} holds.
One can ask whether there is a criterion for strong holomorphicity
expressed in terms of the $\mu_\ell$. 

\begin{thm}\label{chthm}
Assume that $\I$ has pure codimension  $p$ and  that $\mu_\ell$, $\ell=1,\ldots, N$,
generate $\Homs(\Ok/\I,\CH_Z)$. 
Let $\phi$ be meromorphic  and assume that   
\begin{equation}\label{barrolin}
\codim(P_\phi\cap Z_k)\ge k+2, \quad   k>p.
\end{equation}
Then $\phi$ is holomorphic if and only if  
$\phi\mu_\ell$  are $\dbar$-closed for all $\ell$.
\end{thm}

If for instance $\I$ is Cohen-Macaulay, then $Z_k$ is empty for $k>p$ so 
\eqref{barrolin} is fulfilled for any meromorphic $\phi$. 
If $h$ is holomorphic and generically non-vanishing on $Z$, then
$\dbar(1/h)\w \mu_\ell$ are  Coleff-Herrera currents whose common
annihilator is precisely the ideal $h+\I$, see Theorem~\ref{prod} below.

\section{Some residue theory}\label{rester}

In \cite{AW2} we introduced the sheaf of  {\it pseudomeromorphic} 
currents $\PM$ in $X$.
It  is  a module over the sheaf of smooth forms,
and closed under $\dbar$.
For any  $T\in\PM$ and variety $V$ there exists a restriction
$T{\bf 1}_V$  that is in $\PM$  and has support on $V$,  and 
$T=T{\bf 1}_V$ if and only if $T$ has support on $V$. 
Moreover, ${\bf 1}_V{\bf V'} T={\bf 1}_{V\cap V'}T$
and $\xi {\bf 1}_V T= {\bf 1}_V(\xi T)$ if $\xi$ is smooth.
If $H$ is a holomorphic tuple such that
$\{H=0\}=V$, then $|H|^{2\lambda}T$ has a current-valued
analytic continuation to $\Re\lambda>-\epsilon$ and 
\begin{equation}\label{rest}
T{\bf 1}_{V}=T-|H|^{2\lambda}T\big|_{\lambda=0}.
\end{equation}
We say that a current $T$ with support on a variety $V$ has SEP
(with respect to $V$) if $T{\bf 1}_W=0$ for each $W\subset V$ with
positive codimension. 
The following result (Corollary~2.4 in \cite{AW2} will be used frequently.

\begin{prop}\label{hyp}
If $\mu\in\PM$ with  bidegree $(*,p)$ has  support on
a variety $V$ of codimension $k>p$ then $\mu=0$. 
\end{prop}

Let $Z$ be a variety of pure codimension $p$. The sheaf of $\dbar$-closed
$\PM$ currents of bidegree $(0,p)$ with support on $Z$ 
coincides with the so-called 
sheaf of Coleff-Herrera currents, $\CH_Z$;
see Proposition~2.5 in \cite{AW2}.

\smallskip

We have to  recall the construction of a residue current associated with
a complex of locally free sheaves  in \cite{AW1}. Let
\begin{equation}\label{hcomplex}
0\to E_N\stackrel{f_N}{\longrightarrow}\ldots
\stackrel{f_3}{\longrightarrow} E_2\stackrel{f_2}{\longrightarrow}
E_1\stackrel{f_1}{\longrightarrow}E_0\to 0
\end{equation}
be a generically exact complex of Hermitian vector bundles over $X$,
where $E_0\simeq\C$ for simplicity,
let
\begin{equation}\label{complex}
0\to \Ok(E_N)\stackrel{f_N}{\longrightarrow}\ldots
\stackrel{f_1}{\longrightarrow}\Ok(E_0)
\end{equation}
be the corresponding complex of locally free sheaves,
and let $\I$ be the ideal sheaf $f_1\Ok(E_1)\subset\Ok$.
 Assume that
\eqref{hcomplex} is pointwise exact outside the variety $Z$, and 
over $X\setminus Z$  let $\sigma_k\colon E_{k-1}\to E_k$ be the minimal
inverses of $f_k$. Then
$
f\sigma+\sigma f=I,
$
where $I$ is the identity on $E=\oplus E_k$, $f=\oplus f_k$ and $\sigma=\oplus\sigma_k$.
The bundle $E$ has a natural superbundle structure
$E=E^+\oplus E^-$, where $E^+=\oplus E_{2k}$ and $E^-=\oplus E_{2k+1}$,
and $f$ and $\sigma$ are odd mappings with respect to this structure,
see, e.g., \cite{AW1} for more details.

The operator   $\nabla=f-\dbar$ acts as an odd mapping on $\Cu^{0,\bullet}(X, E)$,
the space of $(0,*)$-currents with values in $E$,  and extends to
an odd  mapping $\nabla_\End$ on $\Cu^{0,\bullet}(X, \End E)$,
and $\nabla_{\End}^2=0$.
If 
$$
u=\sigma+(\dbar\sigma)\sigma+(\dbar\sigma)^2\sigma+\cdots,
$$
then  $\nabla_\End u=I$ in $X\setminus Z$. 
One can define a canonical current extension $U$ of $u$
across $Z$ as the analytic continuation  to $\lambda=0$ of
$U^\lambda=|F|^{2\lambda}u$, where $F$ is a holomorphic tuple that
vanishes on $Z$; e.g., $F=f_1$ will do if \eqref{complex} is a resolution. 
From \cite{AW2} we know that $U$ is in $\PM$.
For further reference we notice that ${\bf 1}_VU=0$ for any $V$ with positive codimension.
In fact, since $U$ is smooth outside $Z$,  ${\bf 1}_VU$ must vanish there,
and thus it has support on $Z$. However, from the definition
of $U$ it follows that ${\bf 1}_ZU=0$. Therefore,
${\bf 1}_VU={\bf 1}_Z{\bf 1}_V U={\bf 1}_V{\bf 1}_Z U=0$.
Now 
$$
\nabla_{\End} U^{\lambda}=I-R^\lambda,
$$
where 
\begin{equation}\label{rlambda}
R^\lambda=(1-|F|^{2\lambda})I+\dbar|F|^{2\lambda}\w u.
\end{equation}
Then the  current 
$$
R=R^\lambda|_{\lambda=0}
$$
is in $\PM$, has  support on $Z$, and 
\begin{equation}\label{res}
\nabla_\End U=I-R.
\end{equation}
More precisely, 
$$
R=\sum_{\ell\ge 0}  R^\ell=\sum_{\ell, k\ge 0} R^\ell_k, 
$$
where $R^\ell_k$ is a $\PM$-current of bidegree $(0,k-\ell)$
that takes values in $\Hom(E_\ell,E_k)$.

\smallskip

As before, let $Z_k$ be the set where $f_k$ does not have optimal rank.
By the Buchsbaum-Eisenbud theorem, see \cite{Eis1} Ch.~20,
\eqref{complex} is a resolution of $\Ok/\I$ if and only if
$\codim Z_k\ge k$ for all $k$.
We also recall from \cite{AW1} that
if   \eqref{complex}  is a resolution, then $R^\ell=0$ for all $\ell\ge 1$.
In view of Proposition~\ref{hyp} then  $R=R^0=R_p+R_{p+1} +\cdots$. 
Since $E_0=\C$ we can consider $R=R^0$ as taking values in $E$ rather than $\Hom(E_0,E)$,
and since $\nabla_{\End}R=0$ thus  $\nabla R=0$.

\smallskip

Below we will consider analogues of $R$ and $U$ 
obtained in a different way. The following proposition is proved 
precisely as Proposition~2.2 in  \cite{AW1}.

\begin{prop}\label{22}
Consider the  generically exact complex \eqref{hcomplex} 
and let $U$ and $R$ be any currents such that  \eqref{res} holds. 
If $R^1=0$ then  $\ann R=\I$.
If $R^\ell=0$ for all $\ell\ge 1$ then the 
associated sheaf  complex \eqref{complex}  is exact,
i.e., a resolution of $\Ok/\I$.
\end{prop}

\section{Multiplication by  meromorphic functions}\label{mult}

For any pseudomeromorphic current $T$ and holomorphic function $h$,
the product  $(1/h)T$  is defined in \cite{AW2} (Proposition~2.1)
as the value at $\lambda=0$ of
$|h|^{2\lambda}T$. It is again a pseudomeromorphic current and it is clear that
$\alpha (1/h) T=(1/h) \alpha T$ if $\alpha$ is smooth. However, 
in general it is {\it not} true that $f(1/fg)T=(1/g)T$. 
One can verify, cf., the proof if Proposition~5.1 in  \cite{A12},  that  
$(1/h)T$  is equal to the limit of  $\chi(|h|/\epsilon)T/h$ when $\epsilon\to 0$,
cf.,  \eqref{ep} above. Moreover, if  we define $\dbar(1/h)\w T$ as the value at $\lambda=0$
of $\dbar|h|^{2\lambda}\w(1/h)T$, then the Leibniz rule 
$\dbar[(1/h)T]=\dbar(1/h)\w T+(1/h)\dbar T$ holds.

\begin{lma}\label{anita}
Suppose that $Z\sim \I$ has pure codimension $p$ and let 
$R$ be the residue current associated with a resolution \eqref{acomplex}.
If $h$ is generically non-vanishing on $Z$, then $(1/h)R$
has the  SEP on $Z$. 
\end{lma}

\begin{proof}[Proof of Lemma~\ref{anita}]
Assume that $V\subset Z$ has positive codimension. 
Then $((1/h) R_p){\bf 1}_V=0$ in view of
Proposition~\ref{hyp}. Outside the variety $Z_{p+1}$
we have that $R_{p+1}=\alpha_{p+1}R_p$ where $\alpha_{p+1}=\dbar\sigma_{p+1}$ is smooth,
and hence 
\begin{multline*}
((1/h)R_{p+1}){\bf 1}_V= ((1/h)\alpha_{p+1}R_{p}){\bf 1}_V= \\
(\alpha_{p+1}(1/h)R_{p}){\bf 1}_V=\alpha_{p+1}
((1/h)R_{p}){\bf 1}_V=0.
\end{multline*}
It follows that $((1/h)R_{p+1}){\bf 1}_V$ has support on $Z_{p+1}$ which has codimension
$\ge p+2$, and hence it vanishes by virtue of Proposition~\ref{hyp}.
Now $R_{p+2}=\alpha_{p+2}R_{p+1}$ outside $Z_{p+2}$ that has codimension $\ge p+3$,
and so $(g(1/h)R_{p+2}){\bf 1}_V=0$ by a similar argument. Continuing in this way
the  lemma follows. 
\end{proof}

Given a meromorphic function $\phi$ on $Z$ we can define 
$\phi R$ as $g(1/h)R$ if $g/h$ represents $\phi$. Since
$(1/h)R$ has the SEP also $g(1/h)R$ has.  Since the difference of two
representations of $\phi$ lies in $\I$ outside some $V\subset Z$
of positive codimension and $\I R=0$, it follows from the SEP that
$\phi R$ is well-defined. Moreover, if $\psi\in\Ok_Z$, 
it follows that 
$$
\psi(\phi R)=(\psi\phi) R =\phi (\psi R).
$$
Since $\phi R$ is a well-defined, we also have
a well-defined current $\dbar\phi\w R$, and by the Leibniz rule,
\begin{equation}\label{hepp}
\dbar \phi\w R=-\nabla(\phi R)=g\dbar\frac{1}{h}\w R.
\end{equation}
The proof of Theorem~\ref{thmett}   follows  the outline of the proof of
Tsikh's theorem in the introduction,  and the following result is crucial.

\begin{thm}\label{korre}
Assume that $\I$ has pure codimension and let $R$ be the residue current
associated with a resolution.
If  $h$ is generically non-vanishing
on $Z$, then  the annihilator of 
$$
\dbar\frac{1}{h}\w R.
$$
is precisely  $h+\I$.
\end{thm}

Theorem~\ref{korre} is a special case of a  more general result
for product complexes, Theorem~\ref{prod}, 
that we obtain without too much extra effort.

\begin{remark}
Let $\phi$ be holomorphic in $Z\setminus V$, 
where $V\subset Z$ has positive codimension and contains $Z_{sing}$. 
If $\phi$ is meromorphic on $Z$, then
we have seen that $\phi R$ has a natural current extension from
$X\setminus V$ across $V$. Also the converse holds. 
In fact, one can always find a holomorphic form $\alpha$ with values
in $\Hom(E_p,E_0)$ such that $R_p\cdot\alpha=[Z]$, see \cite{A11}
Example~1.
Therefore, if $\phi R$ has an extension across $V$ also 
$\phi[Z]$ has, and it then follows from \cite{HP} that $\phi$ is meromorphic.
\end{remark}

\section{Tensor products of resolutions}\label{tre}

Assume  that $\Ok(E^g_k), g_k$ and $\Ok(E^h_\ell), h_\ell$ are resolutions of
$\Ok/\I$ and $\Ok/\J$, respectively. 
We can  define a   complex  \eqref{complex}, where
\begin{equation}\label{sven}
E_k=\bigoplus_{i+j=k} E^g_i\otimes E^h_j,
\end{equation}
$f=g+h$, or more formally,
$f=g\otimes I_{E^h}+ I_{E^g}\otimes h$,
such that
$$
f(\xi\otimes\eta)=g\xi\otimes\eta +(-1)^{\deg\xi}\xi\otimes h\eta.
$$
Notice that $E_0=E_0^g\otimes E_0^h=\C$ and that
$f_1\Ok(E_1)=\I+\J$.
One extends \eqref{sven} to current-valued sections $\xi$ and $\eta$ and
$\deg\xi$ then means  total degree. It is  natural to write
$\xi\w \eta$ rather than $\xi\otimes\eta$, and of course
we can define $\eta\w\xi$  as $(-1)^{\deg \xi \deg\eta}\xi\w\eta$.
Notice that 
\begin{equation}\label{rregel}
\nabla (\xi\otimes\eta)=\nabla^g\xi\otimes\eta
+(-1)^{deg \xi}\xi\otimes \nabla^h\eta.
\end{equation}
Let  $u^g$ and $u^h$ be the corresponding
$\Hom(E^g)$-valued and $\Hom(E^h)$-valued 
forms, cf., Section~\ref{rester}.
Then $u=u^h\w u^g$ is a $\Hom(E)$-valued form
outside $Z^g\cup Z^h$. 
Following the proof of  Proposition~2.1 in \cite{AW2} we can define
$\Hom(E)$-valued pseudomeromorphic currents
$$
R^h\w  R^g=R^{h,\lambda}\w  R^g |_{\lambda=0},
\quad 
R^g\w  R^h=R^{g,\lambda}\w  R^h |_{\lambda=0}.
$$

\begin{remark}
It is important here that $R^{h,\lambda}=\dbar|H|^{2\lambda}\w u^h$ with $H=h_1$. If we use
a tuple $H$ that vanish on a larger set than $Z^h$, the result may  be affected.
It is also important to
notice that even if a certain component  $(R^h)^\ell_k$ vanishes, it might very well 
happen that $(R^h)^\ell_k\w R^g$ is non-vanishing. In particular, notice that
$
(R^h)^\ell_\ell\w R^g = {\bf 1}_{Z^h}I_{E^h_\ell}\w R^g,
$
cf., \eqref{rlambda} and \eqref{rest}, which is non-vanishing
if $Z^h\supset Z^g$.
\end{remark}

We can now state our main result of  this section.

\begin{thm}\label{prod}
Assume that  $\I$ and $\J$ are ideal sheaves such that
\begin{equation}\label{dimvillkor}
\codim (Z^\I_k\cap Z^\J_\ell)\ge k+\ell, \quad k,\ell\ge 1. 
\end{equation}
Then
\begin{equation}\label{stake}
R^h\w R^g=R^g\w R^h
\end{equation}
and  the annihilator of $R^h\w R^g$ is equal to $\I+\J$.

In case both sheaves are Cohen-Macaulay and both  resolutions have minimal lengths, 
$R^h\w R^g$  coincides with the current obtained from the tensor product
of the resolutions. 
\end{thm}

\begin{proof}[Proof of Theorem~\ref{korre}]
Let  $\I$ be   the sheaf associated to  $Z$ and let $\J=(h)$. 
Then  $0\to \Ok(E^h_1)\to\Ok(E_0^h)$ is a resolution of
$\Ok/\J$  if  $E_1^h\simeq E_0^h\simeq\C$
and the mapping is multiplication by $h$.
Thus $Z^h=Z^h_1=\{h=0\} $ and $Z^h_\ell=\emptyset$ for $\ell>1$.
Since $Z$ has pure codimension, $\codim Z_k\ge k+1$ for all $k$.
Thus  $\codim Z_k\cap Z^h_\ell\ge k+\ell$.  
Since  $R^h\w R=\dbar(1/h)\w R$, 
Theorem~\ref{korre} follows from Theorem~\ref{prod}.
\end{proof}

\begin{remark}
Let $\I=(g_1)$ and $\J=(h_1)$ be complete intersections, and choose the
Koszul complexes as resolutions. Then, see \cite{AW1},  $R^g$ and $R^h$ are the
Bochner-Martinelli type residues introduced in \cite{PTY}.
Moreover, the tensor product of these resolutions is the Koszul
complex generated by $(g_1,h_1)$,  and so the last statement in the  theorem
means that this product coincides with the Bochner-Martinelli  residue associated with
the ideal $(g_1,h_1)$.  This fact is proved already in \cite{W1}.
\end{remark}

\begin{remark} Theorem~\ref{prod}  extends in a natural way 
to any finite number of ideal sheaves.
\end{remark}

Analogously we can define currents
$$
U^h\w R^g=U^{h,\lambda}\w R^g |_{\lambda=0},
\quad
R^g\w U^h=R^{g,\lambda}\w U^h|_{\lambda=0},
$$
etc. 
From \eqref{rregel} we get that
\begin{equation}\label{ratta}
\nabla_{\End} (U^h\w  R^g)= I^h\w R^g-R^h\w  R^g.
\end{equation}
In fact, $\nabla_{\End} (U^{h,\lambda}\w R^g)=(I^h-R^{\lambda,h})\w R^g$
since $\nabla_{\End}^g R^g=0$ and so
\eqref{ratta} follows. 
In the same way 
\begin{equation}\label{ratta2}
\nabla_{\End}(R^g\w U^h)=R^g\w I^h-R^g\w R^h.
\end{equation}
If we define 
$$
U=I^h\w U^g+U^h\w R^g, \quad   R=R^h\w R^g, \quad I=I_E,
$$
therefore 
\begin{equation}\label{malla}
\nabla_{\End} U=I-R.
\end{equation}

\begin{lma}
If the hypothesis  in Theorem~\ref{prod} holds,
 we have that
\begin{equation}\label{malla2}
U^h\w R^g=R^g\w U^h.
\end{equation}
\end{lma}

\begin{proof}
We have to prove that
\begin{equation}\label{buk}
(U^h)^r_\ell(R^g)^s_k-(R^g)^s_k(U^h)^r_\ell
\end{equation}
vanishes for $\ell>r\ge 0,\ k\ge s\ge 0$.  
Since $U^h$ is smooth outside $Z^h=Z_1^h$,
\eqref{buk} vanishes there. On the other hand, both terms have
support on $Z^h=Z^h_1$. Thus  \eqref{buk} has support
on $Z_1^h\cap Z_1^g$. Let us first consider the
case when $r=s=0$.
If $k=0$, then \eqref{buk} is
$$
0-I^g_{E_0^g}{\bf 1}_{Z^g}(U^h)^0_\ell,
$$
which vanishes since $Z^g$ has positive codimension, cf.,
Section~\ref{rester} above.
Next assume that $\ell=k=1$. Then  \eqref{buk} has bidegree  $(0,1)$
and support on $Z^h_1\cap Z_1^g$, which by the hypothesis
has codimension at least $2$. Thus  \eqref{buk} must vanish
in view of Proposition~\ref{hyp}. We now proceed by induction. Assume that
we have proved that \eqref{buk} vanishes whenever  $\ell+k<m$, and assume that
$\ell+k=m$. 
If $\ell\ge 2$ we know from the induction hypothesis  that
\begin{equation}\label{buk1}
(U^h)^0_{\ell-1}(R^g)^0_k-(R^g)^0_k(U^h)^0_{\ell-1}=0.
\end{equation}
Outside $Z^h_{\ell}$ we can apply the smooth form  $\alpha^h_{\ell}=\dbar\sigma^h_{\ell}$
to \eqref{buk1}, cf., the proof of Lemma~\ref{anita} above,
 and conclude that
\begin{equation}\label{buk2}
(U^h)^0_{\ell}(R^g)^0_k-(R^g)^0_k(U^h)^0_{\ell}
\end{equation} 
vanishes there, i.e., its support is contained in $Z^h_{\ell}$.
If $k\ge 2$ we find in a similar
way that \eqref{buk2} must have support on $Z^g_k$. In any case,
we find that \eqref{buk} has bidegree $(0,m-1)$ and
has support on  $Z^h_\ell\cap Z^g_k$, which  has codimension
at least  $\ell+k=m$, so \eqref{buk}  must vanish.
The case when $r+s>0$ is handled  in a similar way.
\end{proof}

\begin{proof}[Proof of Theorem~\ref{prod}]
Applying $\nabla_{\End}$ to \eqref{malla2} we get by 
\eqref{ratta} and \eqref{ratta2} that
$$
(I^h-R^h)\w  R^g=R^g\w(I^h-R^h)
$$
which is precisely \eqref{stake}.
Since $(R^g)^s=0$ for $s\ge 1$ we have that 
$$
R=\sum_{s,r\ge 0}(R^h)^r\w (R^g)^s
=\sum_{r\ge 0}(R^h)^r\w (R^g)^0.
$$
In view \eqref{stake} we thus have that
$R=(R^h)^0\w (R^g)^0=R^0$
i.e., $R^m=0$ for $m\ge 1$.
From Proposition~\ref{22} we now conclude that 
$\Ok(E),f$ is a resolution and 
$
\ann R=\I+\J.
$

\smallskip
Finally, assume that $\I$ and $\J$ are Cohen-Macaulay sheaves and 
the resolutions $\Ok(E^g),g$ and $\Ok(E^h),h$ have minimal lengths
$\codim \I$ and $\codim\J$, respectively. Then the product
resolution $\Ok(E),f$ has (minimal) length $p=\codim \I+\codim\J$.
Let $U^f, R^f$ denote the currents associated with this complex.
Then $R^f$ as well as $R^h\w R^g$ are $\dbar$-closed pseudomeromorphic
currents of bidegree $(0,p)$ with support on $Z=Z^g\cap Z^h$ which
has codimension $p$, and hence they are Coleff-Herrera currents, according
to Proposition~\ref{hyp}. Moreover, cf., \eqref{malla}, 
$$
\nabla_{\End}(U-U^f)=R^f-R=R^f-R^h\w R^g.
$$
It follows from Lemma~3.1 in \cite{A11} 
that $R^f-R^h\w R^g=0$.
\end{proof}

\begin{remark}
If  $\Ok(E^g), g$ and $\Ok(E^h), h$ are resolutions
one can verify (without residue calculus)  that 
the product complex is a resolution as well if and only if
\eqref{dimvillkor} holds. Since this should be well-known we
just sketch an argument:
It is not too hard to see that (for each fixed point $x$)
\begin{equation}\label{agnes}
H^m(E^h\otimes E^g)=\otimes_{\ell+k=m}H^\ell(E^h)\otimes H^k(E^g).
\end{equation}
In fact,  choose  Hermitian metrics on $E^g$ and $E^h$.
If  $h^*$ and $h^*$  and $f^*=g^*+h^*$ are the induced adjoint mappings and
$\Delta^f=ff^*+f^*f$, etc, then
$\Delta^f=\Delta^g+\Delta^h.$
As usual each  class in $H^m(E^h\otimes E^g)$ has a unique harmonic
representative 
$$
v=\sum_{\ell+k=m}\xi_\ell\w \eta_k.
$$
However, it is easily verified that
$\Delta^f v=0$ if and only if $\Delta^g\xi_\ell=0=\Delta^h\eta_k$
for all $\ell,k$.  Thus \eqref{agnes} follows.

Let $Z_k^\I$ and $Z^\J_\ell$ be the 
varieties associated to the sheaves $\I$ and $\J$.
Since  $\Ok(E^g), g$ is exact, it follows that 
$H^k(E^g)=0$ at a given point $x$ if and only if $x\notin Z^\I_k$
and similarly for $E^h$. In view of \eqref{agnes}, therefore
$H^m(E)\neq 0$ at $x$ if and only if 
$$
x\in \cup_{\ell+k=m} Z^\I_k\cap Z^\J_{\ell}.
$$
Thus  $\codim Z_m\ge m$ for all $m$ if and only if \eqref{dimvillkor}
holds, and according to  the Buchsbaum-Eisenbud theorem
therefore $\Ok(E),f$ is a resolution if and only if \eqref{dimvillkor} holds.
\end{remark}

\section{Proofs of the main results}

We begin with

\begin{proof}[Proof of Theorem~\ref{thmett}]
If $\phi$ is strongly holomorphic, then it is represented by 
a function $\Phi$ that is holomorphic in a \nbh of $Z$.
Thus $\nabla (\phi R)=\nabla (\Phi R)=\Phi \nabla R=0$.

Now assume that $\nabla (\phi R)=0$ and 
$\phi$ is represented by $g/h$. Then by \eqref{hepp}, 
we have that
$$
0=\nabla (g(1/h)R)= -g\dbar\frac{1}{h}\w R.
$$
This means that $g$ annihilates the current $\dbar(1/h)\w R$, and by
Corollary~\ref{korre}  therefore
$g=\alpha h +\psi$,
where $\psi\in\I$. It follows that $\phi$ is represented by
$\alpha$ and thus $\phi\in\Ok_Z$.
\end{proof}

\begin{proof}[Proof of Theorem~\ref{thmtva}]
Assume that $\phi$ is meromorphic and \eqref{kruka} is fulfilled.
Clearly, $\dbar\phi\w R$ has support on $P_\phi\cap Z$, so 
$\dbar\phi\w R_p$ must vanish for degree reasons. If now
$\dbar\phi\w R_k=0$, then it follows that
$\dbar\phi\w R_{k+1}$ has support in $P_\phi\cap Z_{k+1}$, and so
it must vanish for degree reasons.
\end{proof}

\begin{proof}[Proof of Corollary~\ref{ms}]
First assume that $\phi$ is (strongly) smooth and holomorphic on $Z_{reg}$.
It is well-known that each weakly holomorphic function on $Z$
(i.e., $\phi$ holomorphic on $Z_{reg}$ and locally bounded at $Z_{sing}$)
is meromorphic, see, e.g., \cite{Dem}. 
Therefore, we have 
a~priori two definitions of $\phi R$;  either as multiplication of
 smooth function times $R$ or as multiplication by the meromorphic function $\phi$.
However, they coincide on $Z_{reg}$ and by the SEP therefore they
coincide even across $Z_{sing}$.
Therefore also the two possible definitions of 
$\nabla(\phi R)=-\dbar\phi\w R$ coincide. Since $\phi$ is holomorphic on
$Z_{reg}$ it follows that $\dbar\phi\w R$ has support on $Z_{sing}$. On the other
hand, 
$$
(\dbar\phi\w R){\bf 1}_{Z_{sing}}=\dbar\phi\w R {\bf 1}_{Z_{sing}}=0
$$
by Lemma~\ref{anita}, and hence $\nabla(\phi R)=-\dbar\phi\w R=0$. Now
the corollary follows from Theorem~\ref{thmett} with $m=\infty$.
A careful inspection of all arguments reveals that only
a finite number of derivatives (not depending on $\phi$)
come  into play but we omit the details.
\end{proof}


\begin{proof}[Proof of Theorem~\ref{chthm}]
The hypothesis  means that 
$0=\dbar (\phi \mu)$ for all $\mu\in \Homs(\Ok/\I,\CH_Z)$.
It is proved in \cite{A11} (Theorem~1.5) that each current $\mu$  in
$\Homs(\Ok/\I,\CH_Z)$ can be written
$\mu=\xi R_p$ for some $\xi\in\Ok(E^*)$ such that
$f^*_{p+1}\xi=0$ and conversely for 
each such $\xi$ the current $\mu=\xi R_p$
is in $\Homs(\Ok/\I,\CH_Z)$. 
Here $f^*_k$ are the induced mapping(s) on the dual complex $\Ok(E^*_k)$.
Thus 
$$
0=\dbar\phi \w \xi R_p
$$
for each such $\xi$. At a given stalk  outside  $Z_{p+1}$, 
the ideal  $\I_x$ is Cohen-Macaulay, so if we choose a minimal resolution
$\Ok(\tilde E),\tilde f$
there it will have length $p$.  If $\tilde R_p$ denotes the resulting
(germ of a) residue current, then the hypothesis  implies that
$$
0=\dbar\phi\w  \tilde R_p
$$
since then trivially $\tilde f^*_{p+1}\xi=0$ for each $\xi\in\Ok(\tilde E_p^*)$. 
However,  $R_p=\alpha\tilde R_p$, where $\alpha$
is smooth (Theorem~4.4 in \cite{AW1}).
It follows that $\dbar\phi\w  R_p$
vanishes outside $Z_{p+1}$. Since $R_{p+1}=\alpha_{p+1}R_p$ outside $Z_{p+1}$
it follows that also $\dbar\phi\w R_{p+1}$
has support on $Z_{p+1}$.
However, it is clear that $\dbar\phi\w R$ must have support on $P_\phi$.
Using the hypothesis
$\codim(P_\phi\cap Z_k)\ge k+2$ for $k>p$, it follows by induction that
$\dbar\phi\w R=0$.
Thus $\phi$ is strongly holomorphic 
according to Theorem~\ref{thmett}.
\end{proof}

\def\listing#1#2#3{{\sc #1}:\ {\it #2},\ #3.}

\end{document}